\documentclass[11pt,reqno]{amsproc}

\usepackage{amsmath,amsfonts,amssymb,amsthm}
\usepackage[abbrev,lite,nobysame]{amsrefs}
\usepackage{graphics,graphicx}
\usepackage{subcaption}
\usepackage{tabularx}
\usepackage[usenames,dvipsnames]{color}
\usepackage{times}
\usepackage{bbm}
\usepackage[margin=1in]{geometry}
\usepackage[colorlinks=true, pdfstartview=FitV, linkcolor=blue, citecolor=blue, urlcolor=blue]{hyperref}
\usepackage{tikz}
\usepackage{enumerate,enumitem}
\usepackage{mathtools}

\makeatletter

\renewcommand\subsection{\@startsection{subsection}{2}%
\normalparindent{.5\linespacing\@plus.7\linespacing}{-.5em}
{\normalfont\bfseries}}

\renewcommand\subsubsection{\@startsection{subsubsection}{3}%
\normalparindent{.5\linespacing\@plus.7\linespacing}{-.5em}
{\normalfont\bfseries}}

\def\@tocline#1#2#3#4#5#6#7{\relax
  \ifnum #1>\c@tocdepth % then omit
  \else
    \par \addpenalty\@secpenalty\addvspace{#2}%
    \begingroup \hyphenpenalty\@M
    \@ifempty{#4}{%
      \@tempdima\csname r@tocindent\number#1\endcsname\relax
    }{%
      \@tempdima#4\relax
    }%
    \parindent\z@ \leftskip#3\relax \advance\leftskip\@tempdima\relax
    \rightskip\@pnumwidth plus4em \parfillskip-\@pnumwidth
    #5\leavevmode\hskip-\@tempdima
      \ifcase #1
       \or\or \hskip 1em \or \hskip 2em \else \hskip 3em \fi%
      #6\nobreak\relax
    \dotfill\hbox to\@pnumwidth{\@tocpagenum{#7}}\par
    \nobreak
    \endgroup
  \fi}
\makeatother

\newtheorem{theorem}{Theorem}
\newtheorem{proposition}{Proposition}[section]

\theoremstyle{definition}
\newtheorem{definition}[proposition]{Definition}

\newtheorem{remark}[proposition]{Remark}

\numberwithin{equation}{section}

\newcommand\eps{\varepsilon}
\newcommand\e{{\rm e}}
\newcommand\dd{{\rm d}}
\newcommand\ddt{{\frac{\dd}{\dd t}}}
\def\Re{{\rm Re}}

\def\l {\langle}
\def\r {\rangle}
\newcommand\de{{\partial}}

\newcommand{\norm}[1]{\left\lVert #1 \right\rVert}
\newcommand{\jap}[1]{\left\langle #1 \right\rangle}

\newcommand{\ZZ}{\mathbb{Z}}
\newcommand\TT {{\mathbb T}}
\newcommand\RR {{\mathbb R}}

\newcommand\bu{{\boldsymbol u}}

\newcommand\bU{{\boldsymbol U}}

\newcommand\cG{{\mathcal G}}

\newcommand\hOmega{{\widehat \Omega}}

\newcommand\hTheta{{\widehat \Theta}}

\usepackage{mathtools}
\mathtoolsset{showonlyrefs=true}

\newcommand{\red}[1]{\textcolor{red}{#1}}

\def\de{{\partial}}
\def\uu {\boldsymbol{u}}

\usepackage{mathtools}
\mathtoolsset{showonlyrefs=true}

\makeatletter
\def\@xfootnote[#1]{%
	\protected@xdef\@thefnmark{#1}%
	\@footnotemark\@footnotetext}
\makeatother

\title[Symmetrization and asymptotic stability in non-homogeneous fluids]{Symmetrization and asymptotic stability in non-homogeneous fluids around stratified shear flows}
\author[R. Bianchini]{Roberta Bianchini}
\address{IAC, Consiglio Nazionale delle Ricerche, 00185 Rome, Italy}
\email{r.bianchini@iac.cnr.it}
\author[M. Coti Zelati]{Michele Coti Zelati}
\address{Department of Mathematics, Imperial College London, London, SW7 2AZ, UK}
\email{m.coti-zelati@imperial.ac.uk}

\author[M. Dolce]{Michele Dolce}
\address{Institute of Mathematics, EPFL, Station 8, 1015 Lausanne, Switzerland}
\email{michele.dolce@epfl.ch}


\begin{document}

%%%%%%%%%%%%%%%%%%%%%%%%%%%%%%%%%%%%%%
\begin{abstract}
Significant advancements have emerged in the theory of asymptotic stability of shear flows in stably stratified fluids. In this comprehensive review, we spotlight these recent developments, with particular emphasis on novel approaches that exhibit robustness and applicability across various contexts.
\end{abstract}

\maketitle

\setcounter{tocdepth}{1}
\tableofcontents
%%%%%%%%%%%%%%%%%%%%%%%%%%%%%%%%%%%%%%%%%%%%
\section{Introduction}
Under appropriate averaging, a significant portion of Earth's oceans can be regarded as stably stratified fluids. These fluids exhibit dynamic behaviors described by fluctuations around a background density profile that increases with depth - a characteristic known as stable stratification \cite{Buhler2009}. In two-dimensional space and within a specific perturbative regime under the Boussinesq approximation \cite{rieutord2014fluid}, stably stratified fluids can be effectively characterized by what we refer to as the \emph{Boussinesq system around the stratified Couette flow}, written as
\begin{align}
\de_t \bu + y \de_x \bu + (u^y, 0)+(\bu \cdot \nabla) \bu + \nabla q &= -\beta^2 (0,\theta)+ \nu \Delta \bu,  \label{eq:bouss-mom}\\
\de_t \theta + y \de_x\theta + (\bu \cdot \nabla) \theta &= u^y +\kappa \Delta \theta,\label{eq:bouss-den}\\ 
\nabla \cdot \bu&=0,\label{eq:bouss-vel}
\end{align}
where $\bu=(u^x, u^y)$ represents the incompressible velocity field, $q$ is the pressure term ensuring the divergence-free constraint \eqref{eq:bouss-vel}, $\theta$ denotes density, and $\kappa \ge 0$ as well as $\nu \ge 0$ serve as the diffusivity and viscosity coefficients, respectively. The parameter $\beta^2>0$, also known as the Brunt–Väisälä frequency,
represents the maximum frequency of oscillations for internal gravity waves.
While we refrain from providing an exhaustive derivation in this context (for a detailed derivation, see \cite{rieutord2014fluid}), it is essential to note that the system \eqref{eq:bouss-mom}-\eqref{eq:bouss-vel} effectively characterizes the fluctuations around the stratified Couette flow $(\rho_{\mathrm c} (y), \bu_{\text c}(y), q_{\text c}(y))$ given by
\begin{equation}\label{eq:hydro-couette}
\rho_{\mathrm c}(y)=\bar \rho_c -b y, \qquad
\bu_{\mathrm c}(y)=(y, 0),  \qquad q'_{\mathrm c} (y)=- \mathfrak g \rho_{\mathrm c} (y),
\end{equation}
where $\bar \rho_c > 0$ represents the constant averaged density, $b > 0$ is another constant value, and $\mathfrak g$ corresponds to the gravitational constant.
The stratified Couette flow described in \eqref{eq:hydro-couette} is a steady state solution for the incompressible non-homogeneous Euler-Boussinesq equations 
\begin{align}
\bar \rho_c(\de_t+\tilde\uu\cdot  \nabla)\tilde \uu+\nabla  \tilde q&=- \mathfrak g (0, \rho), \label{eq:mom-euler}\\
(\de_t+\tilde \uu\cdot\nabla)\rho&=0, \label{eq:den-euler}\\
\nabla\cdot \tilde \uu&=0,\label{eq:euler}
\end{align}
where $\rho$ is the fluid density. Equipped with the ansatz
\begin{align}\label{eq:couette}
\rho=\rho_{\mathrm c}+ b \theta, \qquad \tilde \bu=(y, 0)+ \bu, \qquad \tilde q=q_{\mathrm c}+\bar \rho_c q,
\end{align}
and defining
\begin{equation}\label{eq:buoyancy-freq}
\beta^2:= \frac{- \rho_{\mathrm c}'(y) \mathfrak g}{\bar \rho_c} = \frac{b \mathfrak g}{\bar \rho_c},
\end{equation}
equation \eqref{eq:mom-euler} directly leads to \eqref{eq:bouss-mom} with $\nu=0$, while \eqref{eq:bouss-den}, with $\kappa=0$, can be similarly derived from \eqref{eq:den-euler}.
In the general case, the momentum equation \eqref{eq:bouss-mom} includes the term $\nu \Delta \bu$, which models typical viscous effects encountered in Navier-Stokes equations. On the other hand, \eqref{eq:bouss-den} incorporates the diffusivity term $\kappa \Delta \theta$, accounting for dissipation attributed to temperature/salinity variations and meso-scale eddies on the large-scale flow \cite{rieutord2014fluid}.
While in numerous geophysical phenomena, both viscosity and diffusivity can be neglected, resulting in the \emph{inviscid} Boussinesq system around Couette (i.e., \eqref{eq:bouss-mom}-\eqref{eq:bouss-vel} with $\nu=\kappa=0$), laboratory conditions often feature characteristic orders of magnitude: $\kappa \sim 10^{-9} m^2/s$ and $\nu \sim 10^{-6} m^2/s$, see \cite{dauxois18}. Consequently, viscous effects can play a significant role.
While some analysis of the fully dissipative case ($\nu,\kappa>0$) and zero-diffusivity case ($\nu>0$, $\kappa=0$) for system \eqref{eq:bouss-mom}-\eqref{eq:bouss-vel} is provided herein, the primary focus of this note lies on the inviscid case, where $\kappa=\nu=0$.

Our analysis focuses on investigating the decay properties of fluctuations $(\bu,\theta)$ around the stratified Couette flow \eqref{eq:hydro-couette}. In Section \ref{sec:linear}, we primarily focus on the linearized system (the linear part of \eqref{eq:bouss-mom}-\eqref{eq:bouss-vel}), and subsequently, in Section \ref{sec:nonlinear}, we delve into the full nonlinear dynamics. It is worth noting that our approach to studying the linearized system can be extended to encompass other hydrostatic steady states beyond \eqref{eq:hydro-couette}. For example, we can handle shear flows close to Couette as demonstrated in \cite{BCZD20}, resulting in a more expansive set of equations compared to \eqref{eq:bouss-mom}-\eqref{eq:bouss-vel}. However, to maintain conciseness, we refrain from delving into such generalizations in this note.
Introducing the vorticity, denoted as $\omega=\nabla^\perp \cdot \bu=-\de_y u^x+\de_xu^y$, we can express equations \eqref{eq:bouss-mom}-\eqref{eq:bouss-vel} in the vorticity-stream formulation
\begin{align}
		\label{eq:omNLintro}\de_t\omega +y\de_x \omega &=-\beta^2\de_x \theta + \nu\Delta\omega-\bu\cdot \nabla \omega,\\
		\label{eq:thNLintro}\de_t\theta+y\de_x \theta &= \de_x\psi+\kappa\Delta\theta-\bu \cdot \nabla \theta, 
\end{align}
where the stream function $\psi$ satisfies the Poisson equation $\Delta \psi=\omega$. We equip 
\eqref{eq:omNLintro}-\eqref{eq:thNLintro}
with initial data
\begin{align}
    \omega(0, x, y)=\omega^{in}(x,y), \qquad 
    \theta(0, x, y)=\theta^{in}(x,y).
\end{align}
The linear part of this system, corresponding to the linearization of the Boussinesq system around the stratified Couette flow \eqref{eq:hydro-couette}, reads
\begin{align}
		\de_t\omega +y\de_x \omega &=-\beta^2\de_x \theta + \nu\Delta\omega,\label{eq:linBoussOm}\\
		\de_t\theta+y\de_x \theta &= \de_x\psi+\kappa\Delta\theta\label{eq:linBoussTh}.
\end{align}
The stability of the stratified Couette flow \eqref{eq:hydro-couette}, or more explicitly, the spectral stability of the linear system \eqref{eq:linBoussTh}, is ensured by the Miles-Howard criterion \cites{miles1961stability, howard1961note}, which requires $\beta^2> 1/4$. Interestingly, in 1975, Hartman \cite{Hartman75} observed an enstrophy Lyapunov instability with a growth rate of $O(t^{\frac12})$, despite the velocity field experiencing a decay rate of $O(t^{-\frac12})$. In recent years, there have been significant contributions to the rigorous mathematical theory of asymptotic stability of the stratified Couette flow (refer to Section \ref{sec:linear} for a precise definition). In this review, we highlight these developments.

The stability results presented here are explicitly quantified in terms of decay rates, characterized by the phenomena known as \emph{inviscid damping} \cite{BM15}, as well as \emph{enhanced dissipation} \cite{CKRZ08}. Please see Theorem \ref{thm:enhancelin} and Section \ref{sec:viscous} for more details. Furthermore, we provide a (sketch of the) proof of the algebraic growth of vorticity $\omega$ and density gradient $\nabla \theta$, referred to as the shear-buoyancy instability, with a growth rate of $O(t^{\frac12})$, thus confirming Hartman's observation \cite{Hartman75}. This instability emerges from the interplay between non-trivial linear dynamics and the presence of the operator $y \partial_x$ (representing the Couette flow) in \eqref{eq:bouss-mom}-\eqref{eq:bouss-vel}.

It is worth noting that enstrophy is conserved when density remains constant, reducing equations \eqref{eq:bouss-mom}-\eqref{eq:bouss-vel} with $\kappa=\nu=0$ to the Euler equations \cite{BM15}. Similarly, (linear) perturbations of the \emph{hydrostatic rest state} $(\rho_{\mathrm c}, (0,0), q_{\mathrm c})$ with $q_{\mathrm c}'(y)=-\mathfrak g \rho_{\mathrm c}(y)$ are described by the equations
\begin{equation}
\begin{aligned}\label{eq:bouss-vel-rest}
\de_t \bu + (u^y, 0) + \nabla q &= -\beta^2 (0,\theta), \\
\de_t \theta - u^y&=0,\\ 
\nabla \cdot \bu&=0,
\end{aligned}
\end{equation}
which exhibit a purely oscillatory dynamics.

We initiate our discussion by addressing the spectral stability of system \eqref{eq:linBoussOm}-\eqref{eq:linBoussTh}. In the subsequent sections, we will demonstrate that the well-known sufficient condition for \emph{spectral stability} (the Miles-Howard criterion) is a crucial component of our approach for proving asymptotic stability. Specifically, it plays a vital role in establishing the time decay of fluctuations in both the linear (and later, the nonlinear) system, within appropriate Sobolev (for the linear case) and Gevrey (for the nonlinear case) spaces.

\subsection{Spectral stability}
The foundations of the mathematical theory of hydrodynamic stability can be traced back to the 19th century, with groundbreaking contributions from luminaries such as Stokes, Helmholtz, Kelvin, Rayleigh, Reynolds, and G. I. Taylor. This field primarily focuses on assessing the stability of fundamental fluid flow patterns, including vortices and parallel flows. For a comprehensive historical overview, the reader can refer to \cites{DRAZIN19661, drazin1981}, and references therein.

Most of the initial investigations in this field revolved around linear stability theory, often referred to as \emph{spectral stability}. In simple terms, this theory seeks to determine whether infinitesimal initial disturbances resembling waves grow exponentially. To illustrate this concept, consider the abstract linear system
\begin{align*}
\de_t \bU = \mathcal{L}_\mathrm{E} \bU,
\end{align*}
where $\mathcal{L}_\mathrm{E}$ is a linear operator on a Hilbert space $H$, arising from the linearization of a nonlinear system around an equilibrium $\bU_\mathrm{E}$ . The notion of spectral stability can be stated as follows.
\begin{definition}[Spectral stability]
Let $\sigma (\mathcal{L}_\mathrm{E})$ be the spectrum of $\mathcal{L}_\mathrm{E}$ in the Hilbert space $H$. The steady state $\bU_\mathrm{E}$ is \emph{spectrally stable} if $\sigma (\mathcal{L}_\mathrm{E}) \cap \{c \in \mathbb{C}: \mathrm{Re}(c)>0 \} = \emptyset$ and it is \emph{spectrally unstable} if  $\sigma (\mathcal{L}_\mathrm{E}) \cap \{c \in \mathbb{C}: \mathrm{Re}(c)>0 \} \neq \emptyset$.
\end{definition}
In the context of homogeneous, incompressible, and inviscid fluids, characterized by equations \eqref{eq:mom-euler}-\eqref{eq:euler} with $\rho=1$ and $\mathfrak g=0$, and considering parallel shear flows denoted as $\bu_\mathrm{E}=(U(y), 0)$, a classical necessary condition for spectral instability arises from the Rayleigh inflection point theorem. This theorem asserts that $U$ is spectrally unstable if it has a non-degenerate inflection point \cite{DRAZIN19661}.

%For the non-homogeneous incompressible Euler equations \eqref{eq:mom-euler}-\eqref{eq:euler}, particularly in the vicinity of a general hydrostatic rest states characterized by the form
%\begin{align}\label{eq:rest}
%(\rho_{\mathrm E}(y), (0,0), q_{\mathrm E}(y)), \qquad q_\mathrm{E}'(y)=-\mathfrak g\, \rho_{\mathrm E}(y),
%\end{align}
%namely system \eqref{eq:bouss-vel-rest} where $\beta^2$ in \eqref{eq:buoyancy-freq} is replaced by the more general \emph{buoyancy frequency} \red{[se abbiamo Boussinesq, quella sotto è $\bar\rho_c$]}
%\begin{equation}\label{eq:equivalence}
%\quad N^2(y):=- \mathfrak g\frac{ \rho_{\mathrm E}'(y)}{\rho_{\mathrm E}(y)},
%\end{equation}
For the Boussinesq equations around the hydrostatic rest state \eqref{eq:bouss-vel-rest}, 
a \emph{necessary and sufficient condition} for spectral stability is the \emph{stable stratification} assumption, i.e. $\rho_{\mathrm c}'=-b<0$. 
%Inserting the ansatz
%\begin{align}\label{eq:ansatz}
%\rho=\rho_{\mathrm E}+\theta, \quad \tilde \bu=\bu, \quad \tilde q = q_{\mathrm E}+q,
%\end{align}
%into \eqref{eq:bouss-vel-rest} and s
Searching for solutions to \eqref{eq:bouss-vel-rest} of the form 
\begin{align}\label{eq:planewave}
\theta(t, x, y)= \theta(y) \e^{st+ikx}, \quad \bu(t,x,y)=\bu (y) \e^{st+ikx}, \quad q(t,x,y)=q(y)\e^{st+ikx}, 
\end{align}
the linearized operator yields the Rayleigh-Taylor equation
\begin{equation}\label{eq:rayl}
-(u^y(y))'' + k^2 \left(1+\frac{\beta^2}{s^2 }\right) u^y(y)=0,
\end{equation}
with $\beta^2$ in \eqref{eq:buoyancy-freq}. 
Through an energy estimate, as detailed in \cite{gallay2019stability}, one can conclude that hydrostatic rest states exhibit spectral stability when $\rho_{\mathrm c}'(y) \le 0$ (i.e. $b\geq0$) and spectral instability when $\rho_{\mathrm c}'(y) > 0$ (i.e. $b<0$). In the stable scenario, the linearized system around hydrostatic rest states exhibits wave dynamics, where $\beta^2$ is indeed the maximum frequency of oscillations for internal gravity waves.

Consider now the case in which the hydrostatic steady state is not at rest, and the background density is an affine function as in \eqref{eq:hydro-couette}, namely 
\begin{align}\label{eq:rest}
(\rho_{\mathrm c}(y), (U(y),0), q_{\mathrm c}(y)), \qquad q_{\mathrm c}'(y)=-\mathfrak g\, \rho_{\mathrm c}(y).
\end{align}
Defining $\gamma(y):=s+ikU(y)$, the ansatz 
\begin{align}\label{eq:sheargen}
\rho=\rho_{\mathrm c}+ b \theta, \quad \tilde \bu=(U(y), 0)+ \bu, \quad \tilde q=q_{\mathrm c}+\bar \rho_c q,
\end{align}
gives
%
%When the hydrostatic steady state is not at rest as \eqref{eq:rest} but it is rather a \emph{stably stratified shear flow}, i.e. \eqref{eq:rest} with $\bu_v{E}=(v(y), 0)$ replacing $(0,0)$ and assuming $\rho_\mathrm{E}=\bar \rho_c -by$ with $b>0$, using the ansatz \eqref{eq:ansatz} and looking for solutions of the form \eqref{eq:planewave} under the Boussinesq approximation leads to
%
\begin{align}
\bar \rho_c(\gamma(y) u^x(y) + U'(y) u^y(y)) &=-ikq(y), \label{eq:i} \tag{i} \\
\bar \rho_c\gamma(y) u^y(y) &= -  q'(y) - \mathfrak g \theta (y), \label{eq:ii} \tag{ii}\\
\gamma(y) \theta (y)&=b u^y(y), \label{eq:iii} \tag{iii}\\
iku^x (y) + (u^y(y))'&= 0 \quad \text{(divergence-free)}. \tag{iv}
\end{align}
Taking now the derivative in $y$ of \eqref{eq:i} and using the divergence-free condition yields
\begin{align*}
-\frac{ik}{\bar \rho_c}  q'(y)&=\gamma'(y) u^x(y) + \gamma(y) (u^x(y))'+ U''(y) u^y(y) + U'(y) (u^y(y))'\\
&= \frac{i\gamma'(y)}{k} (u^y(y))'+ \frac{i\gamma(y)}{k} (u^y(y))''+U''(y) u^y(y) + U'(y) (u^y(y))',
\end{align*}
while multiplying \eqref{eq:ii} by $ik$ and using \eqref{eq:iii}, we deduce that
\begin{align*}
- \frac{ik}{\bar \rho_c}  q'(y)=i k \gamma(y) u^y(y) + i k \mathfrak g \theta(y)= i k \gamma(y) u^y(y) + \frac{i kb}{\bar \rho_c\gamma(y)} {\mathfrak g} u^y(y). 
\end{align*}
Exploiting the above two identities for $q'$ leads to the \emph{Taylor-Goldstein equation} (with the Boussinesq assumption, see \cite{gallay2019stability} for more explanations) %\red{[do we need to repeat with the Boussinesq approximation all the time?]}),
\begin{align}\label{eq:TG}
-(u^y(y))''+ k^2 u^y(y) + \frac{ik}{\gamma(y)} U''(y) u^y(y) + \frac{k^2 \beta^2}{\gamma^2 (y)} u^y(y)=0.
\end{align}
This is exactly Rayleigh-Taylor equation \eqref{eq:rayl} when $v\equiv0$.
Introducing as in \cite{howard1961note} the new variable
\begin{align}
v(y):=\gamma^{-\frac12}(y)u^y(y),
\end{align}
we obtain 
\begin{align}
\gamma^{\frac12} (u^y(y))''=\gamma^{\frac12}\left(\frac{(\gamma v')}{\gamma^{\frac12}}\right)'  + \frac{\gamma^{\frac12}}{2} \left(\frac{\gamma' v}{\gamma^{\frac12}}\right)'=(\gamma v')' + \frac 12(\gamma' v)' - \frac{\gamma'}{2 \gamma} \left(\gamma v'+ \frac 12 \gamma' v\right). 
\end{align}
Multiplying \eqref{eq:TG} by $\gamma^{\frac12}$ and using the above identity yields 
\begin{align*}
-(\gamma v')'- \frac 12 (\gamma' v)' + \frac{\gamma'}{2\gamma} \left(\gamma v'+ \frac 12 \gamma' v\right) + k^2 \gamma v + {ik}U'' v + \frac{k^2 \beta^2}{\gamma} v=0. 
\end{align*}
Taking the product with the complex conjugate $\bar v$, integrating by parts the first addend, exploiting the cancellations and keeping only the real terms, as $(\gamma')^2=-k^2 (U')^2$, we end up with
\begin{align*}
\mathrm{Re}(s) \int_0^1 |v'|^2  + k^2 |v|^2 + \frac{k^2 (U')^2}{|\gamma|^2} \left(\frac{\beta^2}{(U')^2}-\frac 14\right) |v|^2  =0. 
\end{align*}
We thus deduce that $\mathrm{Re}(s) =0$ (spectral stability) under the Miles-Howard criterion \cite{howard1961note}
\begin{align}
\mathrm{Ri}(y):= \frac{\beta^2}{(U'(y))^2} > \frac 14. 
\end{align}
%%%%%
In the case of the stably stratified Couette flow \eqref{eq:hydro-couette}, where $U'\equiv1$, 
the Brunt–Väisälä frequency and the Richardson number coincide, giving
\begin{equation}\label{eq:Rich}
\beta^2 > \frac 14.
\end{equation}
\begin{remark}
The threshold $1/4$ is optimal but the Miles-Howard criterion is only a \emph{sufficient condition} for spectral stability: notice indeed that in the homogeneous case $\rho=1$, any shear flow without inflection point is spectrally stable by Rayleigh's theorem. 
 \end{remark}
 
% Spectral stability is not the only available notion of linear stability.
% 
%
%%%%%%
%\subsection{Asymptotic stability}
%%%%%
\section{Linearized asymptotic stability}\label{sec:linear}

There are several ways to study the longtime dynamics of \eqref{eq:EulerBmove}, some of which involve finding an explicit solution in terms of hypergeometric functions \cite{YL18} 
or Whittaker functions \cites{CZN23strip,CZN23chan}, at least when $\nu=\kappa=0$. The approach developed in \cite{BCZD20} is instead based on an energy method, which well adapts to the nonlinear
setting \cite{BBCZD21}, or the viscous one as we show below.

For any function $f:\TT\times\RR\to\RR$, we use the notation
$$
f_{0}=\frac{1}{2\pi}\int_{\TT}f(x,\cdot)\dd x, \qquad f_{\neq}=f-f_0.
$$
At the linear level, the zero mode in $x$ does not play any significant role in the dynamics of \eqref{eq:linBoussOm}-\eqref{eq:linBoussTh}, as it is simply conserved (when $\nu=\kappa=0$) or diffused (when $\nu,\kappa>0$).
The asymptotic stability in the inviscid case $\nu=\kappa=0$ can be phrased as in the following result. The basic assumption is the spectral stability condition \eqref{eq:Rich}, and the notation is
\begin{theorem}[Linear inviscid damping and instability, \cite{BCZD20}]\label{thm:damplin}
Let $\beta>1/2$, and define
\begin{equation}\label{eq:Cbeta}
C_\beta:= \left[\frac{2\beta+1}{2\beta-1}\exp\left(  \frac{1}{  2\beta-1}\right)\right]^{\frac12}.
\end{equation}
Then there hold the  linear inviscid damping estimates
%	\begin{align}
%	\norm{\theta_{\neq}(t)}_{L^2}+\norm{u_{\neq}^x(t)}_{L^2}\lesssim&\frac{1}{\l t\r^{\frac12}}\left[\norm{\omega_{\neq}^{in}}_{H^{-1}_xL^2_y}+\norm{\theta_{\neq}^{in}}_{L^2_xH^1_y}\right], \qquad \forall t\geq0,\label{eq:main-lin1}\\
%	\norm{u^y(t)}_{L^2}\lesssim& \frac{1}{\langle t \rangle^{\frac32}}\left[\norm{\omega_{\neq}^{in}}_{H^{-1}_xH^1_y}+\norm{\theta_{\neq}^{in}}_{L^2_xH^2_y}\right],  \qquad \forall t\geq0.\label{eq:main-lin2}
%	\end{align}
%	and
%	\begin{equation}
%\label{bd:Lyomega}
%\norm{\omega_{\neq}(t)}_{L^2}+\norm{\nabla\theta_{\neq}(t)}_{L^2}\gtrsim\langle t \rangle^{\frac12}\left[\norm{\omega_{\neq}^{in}}_{L^2_x H^{-1}_y}+\norm{\theta_{\neq}^{in}}_{H^1_x L^2_y}\right], \qquad \forall t\geq0.
%\end{equation}
	\begin{align}
	\norm{\theta_{\neq}(t)}_{L^2}+\norm{u_{\neq}^x(t)}_{L^2}\lesssim\, & C_\beta\l t\r^{-\frac12}\left[\norm{\omega_{\neq}^{in}}_{L^2}+\norm{\theta_{\neq}^{in}}_{H^1}\right],  \label{eq:main-lin1}\\
	\norm{u^y(t)}_{L^2}\lesssim\, & C_\beta \langle t \rangle^{-\frac32}\left[\norm{\omega_{\neq}^{in}}_{H^1}+\norm{\theta_{\neq}^{in}}_{H^2}\right],  \label{eq:main-lin2}
	\end{align}
	and the shear-buoyancy instability estimate
	\begin{equation}
\label{bd:Lyomega}
\norm{\omega_{\neq}(t)}_{L^2}+\norm{\nabla\theta_{\neq}(t)}_{L^2}\gtrsim \frac{1}{C_\beta}\langle t \rangle^{\frac12}\left[\norm{\omega_{\neq}^{in}}_{H^{-1}}+\norm{\theta_{\neq}^{in}}_{L^2}\right],  
\end{equation}
for every $t\geq 0$.
\end{theorem}
The symbols $\lesssim$ and $\gtrsim$ only hide absolute constants, independent of $\beta$.
In stark contrast to the homogeneous Couette flow \cite{BM15},   the system undergoes a Lyapunov \emph{instability} \eqref{bd:Lyomega}.
This can be considered the reason why the decay rates in \eqref{eq:main-lin1}-\eqref{eq:main-lin2} are slower by a factor of $t^{\frac12}$, compared to those that
can be obtained in the constant density case. From a physical viewpoint, density stratification induces creation of vorticity and hence an algebraic growth in $L^2$. Nonetheless, 
density is not simply transported (hence conserved in $L^2$), but decays due to buoyancy mechanism.  

 \begin{table}[h]
    \centering
\begin{tabularx}{\columnwidth}{XX}
\includegraphics[width=\linewidth]{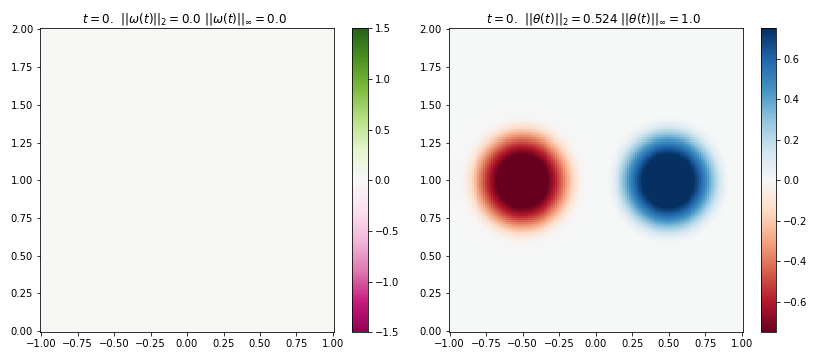}
% \captionof{figure}{Figure A}\label{fig:taba}
    &   \includegraphics[width=\linewidth]{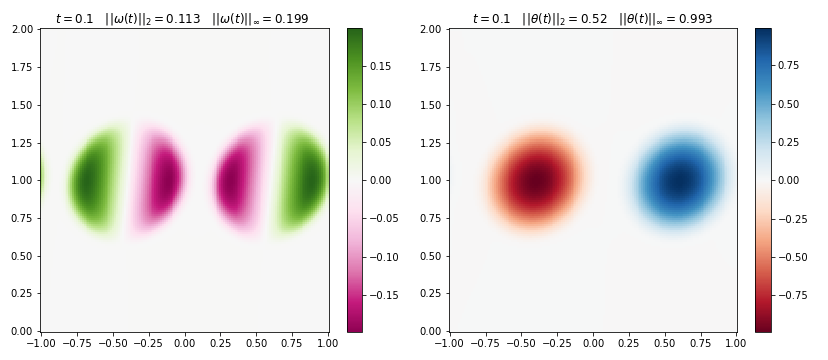}   
        % \captionof{figure}{Figure B}\label{fig:tabb}       
        \\
\includegraphics[width=\linewidth]{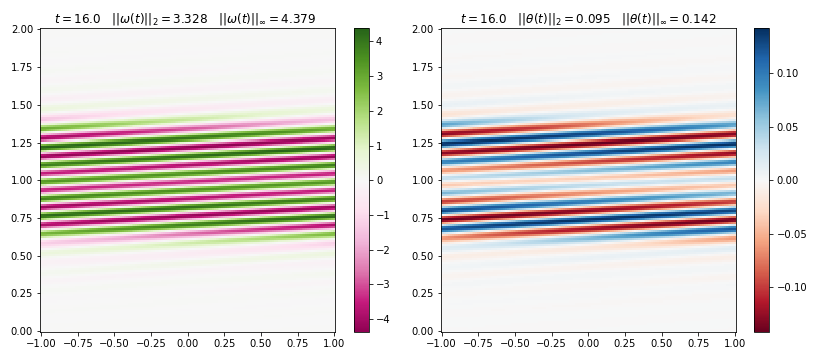}
% \captionof{figure}{Figure C}\label{fig:tabc}
    &  $\qquad$ \includegraphics[width=.8\linewidth]{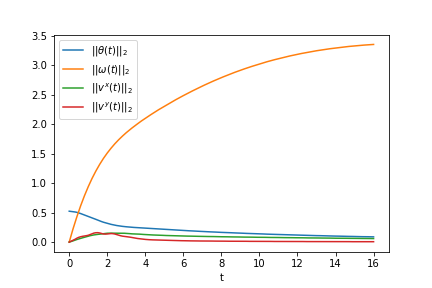} 
        % \captionof{figure}{Figure D}\label{fig:tabd}               
\end{tabularx}
\captionof{figure}{Numerical simulation of the linearized system with $\nu=\kappa=0$. On the top-left the initial data with zero vorticity perturbation (the background vorticity and density are not plotted). On the top-right we see the evolution at time $t=0.1$, where dipoles are created to restore the equilibrium state (heavier density at the bottom). The bottom-left figure shows the evolution at time $t=16$, where the effect of shearing is evident. Observe that the density is fading out at the center, which is in stark contrast compared to the standard advection by the Couette flow. The colors of these three figures are not scaled with magnitude for visualization reasons. Finally, on the bottom-right we see a plot of the $L^2$ norms of density, vorticity and velocity fields components. Notice the $\sqrt{t}$ behavior of the vorticity  (the yellow curve).}
\label{tab:mytable}
    \end{table}

This mechanism affects also the dissipative behavior of the equations, which we firstly analyze  in the case when $\nu,\kappa>0$ are comparable.

%\red{[check position of picture]}
\vspace{5cm}

\begin{theorem}[Linear enhanced dissipation]\label{thm:enhancelin}
Let $\beta>1/2$,  assume that $\nu,\kappa>0$ satisfy
\begin{equation}\label{eq:kappanu}
    \frac{\max\{\nu,\kappa\}}{\min\{\nu,\kappa\}}<4\beta-1,
\end{equation}
and define the strictly positive number
\begin{equation}
\lambda_{\nu, \kappa}:=\min\{\nu,\kappa\}\left(1-\frac{1}{4\beta}-\frac{1}{4\beta}\frac{\max\{\nu,\kappa\}}{\min\{\nu,\kappa\}}\right).
\end{equation}
Then
	\begin{equation}\label{eq:Linenhanced}
\norm{\omega_{\neq}(t)}_{L^2} +\langle t \rangle\norm{\theta_{\neq}(t)}_{L^2} \lesssim C_\beta\langle t \rangle^{\frac12}\e^{- \frac{1}{24 } \lambda_{\nu,\kappa}  k^2t^3 }\left[\norm{\omega_{\neq}^{in}}_{L^2 }+\norm{\theta_{\neq}^{in}}_{H^1}\right], 
\end{equation}
for every $t\geq 0$.
\end{theorem}
Similar estimates on the velocity $(u^x_{\neq}, u^y)$ can be obtained, since the proof of this theorem is based on precise pointwise estimates in the Fourier side. In particular,
when setting $\nu=\kappa$ and then sending them to 0, one recovers all the estimates in Theorem \ref{thm:damplin}.

\begin{remark}
The condition \eqref{eq:kappanu} is the same identified in \cite{CZDZ23} to prove enhanced dissipation in the three-dimensional Boussinesq equations. It reduces to \eqref{eq:Rich} when $\nu=\kappa$.
\end{remark}

\subsection{A new symmetrization scheme}
\label{sec:symm}
The transport structure of system \eqref{eq:linBoussOm}-\eqref{eq:linBoussTh} suggests the linear change of variable
\begin{align}\label{eq:changelin}
z=x-yt,
\end{align}
and the consequent redefinition of the unknowns involved as
\begin{align}
\Omega(t,z,y)=\omega(t,x,y),\qquad \Theta(t,z,y)=\theta(t,x,y),\qquad \Psi(t,z,y)=\psi(t,x,y).
\end{align}
In this moving frame, the Laplace operator becomes $ \Delta_L:=\de_{zz}+(\de_y-t\de_z)^2$ and the equations  \eqref{eq:linBoussOm}-\eqref{eq:linBoussTh}  take the form
\begin{align}\label{eq:EulerBmove}
 \de_t \Omega =-\beta^2\de_z \Theta+ \nu\Delta_L\Omega, \qquad \de_t \Theta=\de_z \Psi+\kappa\Delta_L\Theta .
\end{align}
along with the relation $ \Delta_L\Psi = \Omega$.

Passing to Fourier variables $\TT\times\RR\ni (z,y)\mapsto(k,\eta)\in \ZZ\times \RR$ and denoting the symbol associated to $-\Delta_L$ as
\begin{align}\label{def:p}
	p(t,k,\eta)=k^2+(\eta-k t)^2, \qquad \de_tp(t,k,\eta)=-2k(\eta-k t),
\end{align} 
we arrive at the system
\begin{align}\label{eq:BoussinesqMoveFlin}
		\de_t\hOmega =-i\beta^2k \widehat{\Theta} -\nu p \hOmega,\qquad \de_t\hTheta= -\frac{ik}{p}\hOmega-\kappa p \hTheta.
\end{align}
The idea is to symmetrize this system via time-dependent Fourier multipliers and use an energy functional for the new auxiliary variables
\begin{equation}\label{eq:Z1Z2couette}
Z(t,k,\eta):= \left(p/k^2\right)^{-\frac{1}{4}} \hOmega(t,k,\eta), \qquad Q(t,k,\eta):=\left(p/k^2\right)^\frac{1}{4}i  k\beta\hTheta(t,k,\eta).
\end{equation}
for which  \eqref{eq:BoussinesqMoveFlin} takes the particularly amenable form
\begin{align}
		\de_tZ =-\dfrac{1}{4}\dfrac{\de_tp}{p}Z-|k|\beta p^{-\frac{1}{2}}Q-\nu p Z,\qquad \de_tQ=\dfrac{1}{4}\dfrac{\de_tp}{p}Q+|k|\beta p^{-\frac{1}{2}}Z-\kappa p Q.\label{eq:Qlin}
\end{align}
Point-wise in frequency, define the energy functional 
\begin{align}\label{def:pointwise-functional-Couette}
E(t)=\frac12\left[|Z(t)|^2+|Q(t)|^2+\frac{1}{2\beta}\frac{\de_tp} {|k|p^{\frac12}}  \Re  \left(Z(t) \overline{Q(t)}\right)\right].
\end{align}
Since $|\de_t p/(kp^{\frac12})|\leq 2$, the energy   is coercive for $\beta>1/2$ with
\begin{align}\label{eq:coercive-pointwise}
\frac12\left(1-\frac{1}{2\beta}\right)\left[|Z|^2+|Q|^2\right] \leq E \leq\frac12\left(1+\frac{1}{2\beta}\right)\left[|Z|^2+|Q|^2\right],
\end{align}
and   satisfies the energy identity
\begin{align}\label{eq:energqueq}
\ddt E=\frac{1}{4\beta}\de_t\left(\frac{\de_tp} {|k|p^{\frac12}}\right)\Re\left(Z\overline{Q}\right) - \nu p |Z|^2 - \kappa p|Q|^2
-\frac{\nu+\kappa}{4\beta}\frac{\de_tp} {|k|p^{\frac12}}  p  \Re  \left(Z \overline{Q}\right) .
\end{align}
From this, we can infer the key information needed to prove Theorems \ref{thm:damplin}-\ref{thm:enhancelin}.

\begin{proposition}\label{prop:ellisse}
Let $\beta>1/2$.
If $\nu=\kappa=0$, then for every $t\geq 0$ there holds
\begin{equation}\label{eq:energbelowabove}
\frac{1}{C^2_\beta}\left[|Z(0)|^2+|Q(0)|^2\right] \leq |Z(t)|^2+|Q(t)|^2 \leq C^2_\beta\left[|Z(0)|^2+|Q(0)|^2\right],
\end{equation}
where $C_\beta$ is given in \eqref{eq:Cbeta}.
If $\nu,\kappa>0$ comply with \eqref{eq:kappanu},then  there holds 
\begin{equation} \label{eq:energenhanced}
|Z(t)|^2+|Q(t)|^2\leq C^2_\beta\e^{- \frac{1}{12 } \lambda_{\nu,\kappa}  k^2t^3 } \left[|Z(0)|^2+|Q(0)|^2\right] ,
\end{equation}
for every $t\geq 0$.
\end{proposition}
\begin{proof}
When $\nu=\kappa=0$, identity \eqref{eq:energqueq} and the coercivity \eqref{eq:coercive-pointwise} imply that 
\begin{equation} 
-\frac{1}{2(2\beta-1)}\de_t\left(\frac{\de_tp} {|k|p^{\frac12}}\right) E  \leq \ddt E\leq \frac{1}{2(2\beta-1)}\de_t\left(\frac{\de_tp} {|k|p^{\frac12}}\right) E.
\end{equation}
Integrating the inequality, using again that $|\de_t p/(kp^{\frac12})|\leq 2$ and the coercivity \eqref{eq:coercive-pointwise},  we obtain precisely \eqref{eq:energbelowabove}.
Now, when $\nu,\kappa>0$, we observe that
\begin{align}
\frac{\nu+\kappa}{4\beta}\frac{\de_tp} {|k|p^{\frac12}}  p  \Re  \left(Z \overline{Q}\right) \leq \frac{\nu+\kappa}{4\beta} p\left[|Z|^2+|Q|^2\right],
\end{align}
Thanks to \eqref{eq:kappanu}, the constants 
\begin{equation}
\lambda_\nu:= \nu -  \frac{\nu+\kappa}{4\beta} , \qquad \lambda_\kappa:= \kappa -  \frac{\nu+\kappa}{4\beta}, \qquad \lambda_{\nu, \kappa} =\min \{\lambda_\nu,\lambda_\kappa\}
\end{equation}
are all strictly positive, 
and the energy identity \eqref{eq:energqueq} implies
\begin{align} 
\ddt E
\leq \frac{1}{2(2\beta-1)}\de_t\left(\frac{\de_tp} {|k|p^{\frac12}}\right) E- \frac{4\beta}{2\beta +1 }\lambda_{\nu,\kappa} p E,
\end{align}
so that
\begin{align} 
E(t)\leq \exp\left(\frac{1}{ 2\beta-1 } \right)\exp\left(- \frac{\beta}{3(2\beta +1) }\lambda_{\nu,\kappa} k^2t^3 \right)E(0) ,
\end{align}
and \eqref{eq:energenhanced} follows.

\end{proof}

The proof of Theorems \ref{thm:damplin}-\ref{thm:enhancelin} now follows from reverting back to the original variables (see \cite{BCZD20} for details). 
The decay/growth rates are encoded in the weights $p^{\pm\frac14}\sim \langle t\rangle^{\pm\frac12}$
of the symmetric variables \eqref{eq:Z1Z2couette}. Theorem \ref{thm:damplin} was proved in \cite{BCZD20} in the more general case when the background flow is 
a shear flow close to the Couette flow. Theorem \ref{thm:enhancelin} is essentially contained in \cite{ZillingerBouss}, without any restriction on $\beta$.
The proof here is different and captures 
both inviscid damping and the sharp algebraic growth of the vorticity, at the cost of requiring the spectral stability condition at the inviscid level.

\begin{remark}
The symmetrization scheme above is very robust: it is applicable in various contexts, such as compressible fluids \cite{antonelli2021linear}, magneto-hydrodynamics \cite{Dolce23}, and three-dimension Boussinesq \cite{CZDZ23}. As we shall see later, capturing the optimal energy structure leads to quantitative improvements on the related nonlinear transition thresholds problems \cites{Dolce23, ZZ23}. We believe that, in fact, this procedure works in other coupled systems arising from geophysical/astrophysical models involving the Couette flow.
\end{remark}

\begin{remark}
    There is an intuitive explanation for the definition of the symmetric variables and their associated energy functional \eqref{eq:energqueq}. The weights are determined in the following way:
    \begin{enumerate} [label=(\alph*), ref=(\alph*)]
\item\label{pointa} impose that the anti-symmetric terms are equal in modulus but with opposite signs after weighting;
\item\label{pointb} require that the weight for one variable is the inverse of the weight of the other variable (up to constants).
\end{enumerate}
     Point \ref{pointa}  ensures the cancellation of antidiagonal terms in energy estimates. Point \ref{pointb} 
     allows us to eliminate non-integrable error terms resulting from the non-commutation between weights and time derivatives. This is achieved by introducing mixed terms in the energy functional.
       
    To derive the energy functional $E$, one can initially  assume that all the factors multiplying the variables are time-independent. Then we can choose the mixed term so that the energy is coercive and constant, namely showing that the constant coefficient dynamics lies on an ellipse in phase-space. 
\end{remark}

\begin{remark}
As the limit $\beta\to\infty$ is approached, buoyancy forces dominate and the influence of the background shear is negligible. Consequently, the constant $C_\beta$ converges to 1, indicating the conservation of symmetric variables. This limit effectively restores the model described by \eqref{eq:bouss-vel-rest}, representing the case of zero background flow \eqref{eq:bouss-vel-rest}.
\end{remark}

\subsection{The zero-diffusivity case}
\label{sec:zerodiff}
The condition \eqref{eq:kappanu} on the comparability of the viscosity and the diffusivity is crucial for the proof of Theorem \ref{thm:enhancelin}. 
However, a physically relevant case is when $\kappa=0$ and $\nu>0$, where $\theta$ really plays the role of the density of an incompressible, stratified, viscous 
fluid. The problem has been address in \cite{masmoudi20bouss}, where a careful analysis of the full nonlinear system has been carried out for Gevrey-regular initial data.
At the linearized level, the result of \cite{masmoudi20bouss} can be phrased as follows.

\begin{theorem}[Stability in the non-diffusive case, \cite{masmoudi20bouss}]\label{thm:nondiffmasmoudi}
Let $\beta,\nu>0$ and $\kappa=0$ in \eqref{eq:linBoussOm}-\eqref{eq:linBoussTh}.
Then there hold the asymptotic stability estimates
	\begin{equation}
	\norm{\omega_{\neq}(t)}_{L^2}+\l t\r\norm{u_{\neq}^x(t)}_{L^2}+\l t\r^{2}\norm{u^y(t)}_{L^2}\lesssim  \l t\r^{-2}\left[\norm{\omega_{\neq}^{in}}_{H^4}+\norm{\theta_{\neq}^{in}}_{H^5}\right],
\end{equation}
and 
\begin{equation}
\norm{\theta_{\neq}(t)}_{L^2}\lesssim  \norm{\omega_{\neq}^{in}}_{H^2}+\norm{\theta_{\neq}^{in}}_{H^1},
\end{equation}
for every $t\geq 0$.
\end{theorem}
The  crucial differences with the previous cases is that $\theta$ does not decay and, despite dissipation, $\omega$ only decays algebraically. It is worth noticing that the
above theorem does not claim any enhanced dissipation, as the constants hidden by $\lesssim$ depend (in a quite bad way, in fact) on $\nu$. Interestingly, no condition on $\beta$
is required.

In the usual moving frame \eqref{eq:changelin}, the linearized system \eqref{eq:BoussinesqMoveFlin} reads
\begin{align}\label{eq:BoussinesqMoveFlinZerodiff}
		\de_t\hOmega =-i\beta^2k \widehat{\Theta} -\nu p \hOmega,\qquad \de_t\hTheta= -\frac{ik}{p}\hOmega .
\end{align}
The analysis of such system in nontrivial, as in the equation for $\hOmega$ there is a competition between the buoyancy force and the viscosity. 
This aspect is captured by the \emph{good unknown} 
\begin{equation}\label{eq:sigmadef}
\Sigma:= -\beta^2 ik \hTheta -\nu p\hOmega,
\end{equation}
which satisfies the equation
\begin{equation}
\de_t \Sigma + \nu p \Sigma= -\frac{\beta^2k^2}{p}\hOmega+\nu\de_tp\hOmega.
\end{equation}
Writing the system for $(\Sigma, \hTheta)$, and performing an energy estimate on $|\Sigma|^2+|\hTheta|^2$, we find that
\begin{equation}
|\Sigma(t)|^2+|\hTheta(t)|^2\lesssim_{\nu,\beta} |\Sigma(0)|^2+|\hTheta(0)|^2,
\end{equation}
and therefore \eqref{eq:sigmadef} entails the algebraic decay estimate
\begin{equation}
|\hOmega (t)|\lesssim  \frac{1}{ k^2+(\eta-k t)^2}\left[|\Sigma(0)|+|k\hTheta(0)|\right].
\end{equation}
From this pointwise estimate in frequency space, it is not hard to obtain Theorem \ref{thm:nondiffmasmoudi}.

\section{Inviscid nonlinear stability}\label{sec:nonlinear}
Now we consider the full nonlinear inviscid system, i.e. \eqref{eq:omNLintro}-\eqref{eq:thNLintro} with $\nu=\kappa=0$. 
 From the linearized problem studied in the previous section, we know that the behavior of $\bu_0$ and $\bu_{\neq}$ is fundamentally different: for the $x$-average we do not have any inviscid decay mechanism acting, whereas for $\bu_{\neq}$ we have the inviscid damping estimates in Theorem \ref{thm:damplin}. However, thanks to $\nabla \cdot \bu=0$, we know that $\bu_0=(u_0(t,y),0)$, meaning that this part of the velocity field is a \textit{time-dependent shear flow}. It is therefore natural to rewrite the equations \eqref{eq:omNLintro}-\eqref{eq:thNLintro}  as 
\begin{align}
\de_t\omega+(y+u_0(t,y))\de_x\omega &=-\beta^2\de_x\theta-\bu_{\neq}\cdot \nabla \omega,\\
\de_t\theta+(y+u_0(t,y))\de_x\theta &=\de_x\psi-\bu_{\neq}\cdot \nabla \theta.
\end{align}
For an initial perturbation of size $O(\eps)$, we might hope to treat $u_0$ perturbatively at most on a time-scale $O(\eps^{-1})$, which is the natural time-scale on which the linearized dynamics can be thought to be a good approximation. For longer times, we need to handle the zero $x$-mode differently: we remove it through a \textit{nonlinear} change of coordinates by defining
\begin{equation}
\label{eq:nonlinearchange}
v=y+\frac{1}{t}\int_0^tu_0(\tau,y)\dd \tau, \qquad z=x-vt
\end{equation}
This was indeed one of the key ideas introduced in \cite{BM15} for the 2D Euler equations.  This approach can be viewed as a hybrid perspective combining the Lagrangian and the Eulerian frameworks. The change of coordinates follows on the characteristics determined by the leading order component of the velocity (at long times), provided that inviscid damping is indeed sustained.

On the other hand, due to the instabilities of the linearized problem, we expect that $\omega\approx O(\sqrt{t}\eps)$. Therefore, the best time-scale on which one can hope to handle $u_{\neq}\cdot \nabla \omega$ perturbatively is $t\ll \eps^{-2}$. This is precisely quantified in the main theorem in \cite{BBCZD21}, obtained in collaboration with Bedrossian. We present below a simplified version of it, in order to highlight the analogies with Theorem \ref{thm:damplin}. We need the following notation for the statement: for any function $f:\TT \times \RR\to \RR$ and $0<s\leq 1$, we define the Gevrey--$1/s$ norm as 
\begin{equation}
\norm{f}_{\cG^{\lambda,s}}^2:=\sum_{k\in \ZZ}\int_{\RR}\e^{2\lambda(|k|+|\eta|)^s}|\hat{f}(k,\eta)|^2\dd \eta.
\end{equation}
\begin{theorem}[Nonlinear inviscid damping and instability, \cite{BBCZD21}]
\label{thm:main}
Let $\beta>1/2$. For all $s\in (1/2,1]$ and $\lambda_0>0$, there exists $\delta=\delta(\beta,s,\lambda_0)\in (0,1)$ and $\eps_0=\eps_0(\beta,s,\lambda_0)\in (0,\delta)$ such that the following holds true: let $\eps\leq \eps_0$ and assume that 
\begin{equation}
\label{eq:inGevrey}
\norm{\bu^{in}}_{L^2}+\norm{\omega^{in}}_{\cG^{\lambda_0,s}}+\norm{\theta^{in}}_{\cG^{\lambda_0,s}}\leq \eps.
\end{equation}
Then, for all $0\leq t\leq \delta \eps^{-2}$, there hold the nonlinear inviscid damping estimates 
\begin{equation}
    \norm{u^x_{\neq}(t)}_{L^2}+\norm{\theta_{\neq}(t)}_{L^2}+\jap{t}\norm{u^y(t)}_{L^2}\lesssim \frac{\eps}{\jap{t}^\frac12}.
\end{equation}
Moreover, there exists $K=K(\beta,\lambda_0,s)$ such that if 
\begin{equation}
\norm{\omega_{\neq}^{in}}_{H^{-1}}+\norm{\theta_{\neq}^{in}}_{L^2}\geq K\eps \delta,
\end{equation}
the shear-buoyancy instability estimate
\begin{equation}
    \norm{\omega_{\neq}(t)}_{L^2}+\norm{\nabla \theta_{
    \neq}}_{L^2}\approx \jap{t}^{\frac12}\eps 
\end{equation}
holds true for all $0\leq t\leq \delta \eps^{-2}$.
\end{theorem}
In the above result, it is evident that the linearized behavior encoded in Theorem \ref{thm:damplin} persists (qualitatively) throughout the entire time interval $[0,\delta\eps^{-2}]$.
 The $\eps^{-2}$ time-scale is significantly longer compared to the $\eps^{-1}$ range within which linear properties are typically expected to propagate.
However, two crucial distinctions are noteworthy:
    \begin{enumerate} [label=(\roman*), ref=(\roman*)]
    \item\label{item:reg} The initial datum is  extremely smooth, as precisely quantified by \eqref{eq:inGevrey};
    \item\label{item:nonlinearframe} In the nonlinear setting, we control the variables in the nonlinear moving frame \eqref{eq:nonlinearchange}.
\end{enumerate}
For what concerns \ref{item:reg}, this smoothness is linked to \textit{derivative losses} within the system. These losses are quite severe and connected to an \textit{inverse cascade mechanism}. The use of the Gevrey class is now a standard approach in such problems, as seen in \cite{BM15}, and in the 2D Euler case, it has been demonstrated that $s=1/2$ is the optimal choice to prove asymptotic stability of the vorticity (see \cite{deng2018long}). Further details on the mechanism behind derivative losses are provided in Section \ref{sec:toy}, where the expected worst case scenario is explained. Once this is understood, in Section \ref{sec:energies} we discuss some properties of the weights needed to to carry out the energy estimates, upon which the proof of Theorem \ref{thm:main} is based.

We do not discuss the major technical difficulties related to \ref{item:nonlinearframe}, primarily concerning energy estimates at the highest level of required regularity. These challenges go beyond the scope of this introductory note to the problem. Consequently, we will ignore the $x$-averages in the sequel. Indeed, our strategy and proof ideas are rooted in the intuition developed from the linearized problem and toy models excluding $u_0$.

\begin{remark}[On the physical significance of the zero mode]
The presence of $u_0$ not only presents a technical challenge but also entails crucial physical aspects that must not be overlooked.
In the 2D Euler case, having that $\bu_{\neq}\to 0$ strongly in $L^2$  implies that the kinetic energy of the full solution\footnote{In $\TT\times \RR$, the Couette flow has infinite energy, but this line of reasoning easily adapt to the finite channel where similar results are true at least for perturbations supported away from the boundary.}, a conserved quantity, must be concentrated all in the shear flow $y+u_0^{\infty}(y)$, where $u_0^{\infty}$ is the infinite-time limit of the zero-th mode. This phenomenon is a consequence of the inverse cascade in 2D fluids, where, in order to preserve kinetic energy, information is transferred from small spatial scales to larger ones, in this case represented by the zero mode. 

In the Euler-Boussinesq equations \eqref{eq:mom-euler}-\eqref{eq:den-euler}, the conserved quantity is the sum of kinetic and potential energy $$\frac12\int|\tilde{\boldsymbol{u}}|^2\dd x\dd y +\frac{\mathfrak{g}}{\bar{\rho}_{\text c}}\int  \theta y\; \dd x\dd y.$$ 
Since both density and velocity undergo inviscid damping, and the above quantity is conserved, they cannot be concentrated all at high $x$-frequencies. 
Therefore, in the Boussinesq case, an additional challenge arises: the role of $\theta_0$ becomes fundamental in the dynamics. This aspect introduces further technical difficulties, which we briefly outline in Section \ref{sec:energies}. 
\end{remark}
 \subsection{A toy model for the derivative loss mechanism}
 \label{sec:toy}
Consider the nonlinear term associated to $\bu_{\neq}\cdot \nabla \omega$.  Assuming that $u_0=0$, following the notation introduced in Section \ref{sec:symm}, we first notice a crucial cancellation in the nonlinear term: for any $G$ we have 
\begin{equation}
\bU_{\neq}\cdot \nabla_LG=\nabla_L^\perp \Delta_L^{-1}\Omega_{\neq}\cdot \nabla_LG=\nabla^\perp \Delta_L^{-1}\Omega_{\neq}\cdot \nabla G.
\end{equation}
Namely, some factors of time cancel out and at first glance, one might hope that this term is integrable over time, considering that $\Delta_L^{-1}$ scales as $t^{-2}$. This is of course hopeless due to the loss of regularity associated with achieving such decay. Consequently, we must gain a more precise understanding of the structure of the nonlinear term.

 Inspired by the linearized problem, the good quantities to control are the symmetric variables $Z,Q$ (rigorously defined now with the change of coordinates \eqref{eq:nonlinearchange}). When we retain only the nonlinear term in the equation for $Z$, a first approximation is
\begin{equation}
\de_t Z= (k^2/p)^\frac14\mathcal{F}(\nabla^\perp \Delta_L^{-1}\Omega_{\neq}\cdot \nabla \Omega).
\end{equation}
In the Fourier space, we know that 
\begin{equation}
\label{eq:app1}
|\mathcal{F}(\nabla^\perp \Delta_L^{-1}\Omega_{\neq})|\leq \frac{|\eta|+1}{k^2}\frac{1}{1+(\frac{\eta}{k}-t)^2}|\hat{\Omega}(k,\eta)|=\frac{|\eta|+1}{k^2}\frac{1}{(1+(\frac{\eta}{k}-t)^2)^\frac34}|Z(k,\eta)|
\end{equation}
When $|\eta|\gg k^2$ and $t$ is close to $\eta/k$ (the Orr's critical time \cite{BM15}), we see that this term can be extremely large. Being the equation nonlinear, the term in \eqref{eq:app1} might as well excite neighbouring frequencies leading to potentially enormous growth in the system due to an inverse cascade mechanism. Indeed, the critical time for the mode $k$ is $t_k=\eta/k$, which is before the one for the mode $k-1$, that is $t_{k-1}=\eta/(k-1)$.  To mimic this possible inverse cascade, the worst possible case is when $\nabla^\perp \Delta_L^{-1}\Omega_{\neq}$ is at very high-frequencies whereas $\nabla \Omega$ is at low frequencies. In the opposite regime, we can hope to pay regularity on the velocity field in order to get back integrability in time. When the two terms are at comparable frequencies, it is enough to allow some continuous, but finite, loss of derivatives in time.  Therefore, we will consider only the term $\de_v\Delta^{-1}_L\Omega_{\neq}\de_z\Omega$ (associated with the $|\eta|/k^2$ growth). The next (brutal) approximation is to assume that $\de_z\Omega$ only lives at frequencies $k=\pm1,\eta=0$ so that we can write explicitly  
\begin{align}
\de_tZ(k,\eta)&= \frac{|k|^\frac12}{p(k,\eta)^\frac14}\sum_{\ell=k\mp 1} \frac{\pm\eta}{\ell^2+(\eta-\ell t)^2}\hat{\Omega}(\ell,\eta) \hat{\Omega}(\pm 1,0) \\
&=\frac{1}{(1+(t-\eta/k)^2)^\frac14}\sum_{\ell=k\mp 1} \frac{\pm\eta}{\ell^2(1+(t-\eta/\ell)^2)^\frac34}(1+t^2)^\frac14Z(\ell,\eta) Z(\pm1,0),
\end{align}
where in the last line we used the definition of $Z$ in \eqref{eq:Z1Z2couette}. We are thus left with a system of ODEs that is approximating the interactions of the frequency $k$ with the neighbouring ones $k\pm 1$. We restrict our attention to the time interval $|t-\eta/k|\leq |\eta|/k^2$, which is when the term in \eqref{eq:app1} is dangerous. To get an easy ODE to study, we consider only the interactions between $k$ and $k-1$ (to mimic the inverse cascade) and we assume that $Z(\pm 1,0)=\eps$. Since $|t-\eta/k|\leq |\eta|/k^2$, we know that $|t-\eta/(k-1)|\gtrsim |\eta|/k^2$. Hence, ignoring signs and assuming $\eta>0$ we arrive at the coupled system of ODEs
\begin{align}
\label{eq:toy1}
\de_t Z_k&=\left(\frac{k^2}{\eta}\right)^\frac12 \frac{\eps \sqrt{1+t}}{(1+(t-\eta/k)^2)^\frac14} Z_{k-1},\\
\label{eq:toy2}\de_t Z_{k-1}&=\left(\frac{\eta}{k^2}\right)^\frac12 \frac{\eps \sqrt{1+t}}{(1+(t-\eta/k)^2)^\frac34} Z_{k},
\end{align}
that has to be considered for $|t-\eta/k|\leq \eta/k^2$.
\begin{remark}
The approximations done to obtain this toy model are analogous to the ones in the 2D Euler case \cite{BM15}. However, here the relations between the $k$ and $k-1$ modes are different:  in the equation for $Z_k$ we have a time-dependent factor in front of $Z_{k-1}$, whereas in the 2D Euler case only some power of $k^2/\eta$ is present. This is related to the fact that the symmetric variable $Z$ requires less regularity than $\Omega$. Moreover, notice the term $\eps \sqrt{1+t}$ at the numerator. This is a first clear indication of why we need to work on a time-scale $t=O(\eps^{-2})$, because otherwise we  have some exponential growth (hence loss of regularity since $t\approx \eta/k$) that will break down any nonlinear perturbative argument. On the other hand, for $t\lesssim \eps^{-2}$, thanks to some standard energy estimate we are able to prove that if $|Z_k|=|Z_{k-1}|=1$ then 
\begin{equation}
|Z_k(t)|+|Z_{k-1}(t)|\lesssim (\eta/k^2)^c, \qquad |t-\eta/k|\leq \eta/k^2, \qquad t\lesssim \eps^{-2},
\end{equation}
for some constant $1<c<4$. That is, we have at most an algebraic loss of derivatives in the time-interval $|t-\eta/k|\leq \eta/k^2$.
\end{remark}
Starting at time $t=\eta/k$ and investigating what happens at $\eta/(k-1)$ and so on, the behavior resembles that of the picture below 
\medskip

\begin{center}
\begin{tikzpicture}
\draw[thick] (0,0) -- (9,0);
\draw[thick,dashed] (9,0) -- (12,0); 
\draw[thick,->] (12,0) -- (15,0)node[anchor=north east] {$t$};
%\draw[thick,->] (0,0) -- (14,0)node[anchor=north east] {$t$};
\draw[thick] (1,-.1)--(1,.1)  node[anchor=south] {$Z_k=1$} ;
\node[anchor=north] at (1,-.2) {$\frac{\eta}{k}$};

\draw[thick] (4,-.1)--(4,.1)  node[anchor=south] {$ Z_{k-1}=\left(\frac{\eta}{k^2}\right)^c$} ;
\node[anchor=north] at (4,-.2) {$ \frac{\eta}{k-1}$};

\draw[thick] (8,-.1)--(8,.1)  node[anchor=south] {$Z_{k-2}=\left(\frac{\eta}{k^2}\frac{\eta}{(k-1)^2}\right)^c$} ;
\node[anchor=north] at (8,-.2) {$\frac{\eta}{k-2}$};

\draw[thick] (13,-.1)--(13,.1)  node[anchor=south] {$Z_{1}=\left(\frac{\eta^k}{k!^2}\right)^c$} ;
\node[anchor=north] at (13,-.2) {$\eta$};
\end{tikzpicture}
\end{center}
This cascade is only relevant for $|k|\leq \sqrt{\eta}$, hence, thanks to Stirling's approximation we obtain that the worst possible growth on $Z_1$ is of order 
\begin{equation}
\label{bd:loss}
Z_1(\eta) \approx \left(\frac{\eta^k}{k!^2}\right)^c\lesssim \frac{1}{\sqrt{\eta}}\e^{-c_0\sqrt{\eta}},
\end{equation}
suggesting the Gevrey-2 regularity loss.

\subsection{The weighted energy functional}
Having a good guess for the worst case scenario in terms of growth/loss of regularity, the idea is to first construct a Fourier multiplier $w(t,k,\eta)$ mimicking the behavior of the toy model \eqref{eq:toy1}-\eqref{eq:toy2}. Its definition is rather technical, but let us explain the main features needed. First, we want $w$ to be piecewise smooth and increasing in time so that when we study 
\begin{equation}
\label{eq:dtwz}
\de_t (w^{-1} Z)=-\frac{\de_t w}{w} (w^{-1} Z)+\text{other terms},
\end{equation}
we gain some damping term from $-\de_t w/w<0$. This term will help us absorb the nonlinear errors related to the toy model. However, it is important to note that it also \textit{weakens} our norms, leading to a continuous loss of derivatives over time, which is a standard argument in Cauchy-Kovalevskaya type results.

In our case, for fixed frequencies $k,\eta$, we need to distinguish carefully the behavior near \textit{resonant} times $|t-\eta/k|\leq \eta/k^2$ and outside. We expect that on each time interval of size $\eta/\ell^2$ with $\ell=1,\dots \sqrt{\eta}$ the $k$-th mode experiences growths of order some powers of $\eta/k^2$. But we have to study also the difference between the growth on resonant and non-resonant time, a feature that reflects in an \textit{imbalanced} weight both in time and frequencies. Moreover, on account of \eqref{bd:loss}, it must allow for a loss of derivatives of order $\e^{\sqrt{\eta}}$. The full construction can be found in \cite{BBCZD21}*{Section 4}, but it behaves heuristically as in the drawing below
              \begin{center}
             \begin{tikzpicture}[scale=.7]
            	\draw[->] (0,0)--(11.5,0);
            	\draw[->] (0,0)--(0,6);
            	\draw[dashed,red] (0,.25)--(11,.25);
            	\draw[dashed,red] (7.5,.25)--(7.5,1.25);
            	\draw[dashed,red] (6,1.25)--(7.5,1.25);
            	\draw[dashed,red] (6,1.25)--(6,3.25);
            	\draw[dashed,red] (4,3.25)--(6,3.25);
            	\draw[dashed,red] (4,3.25)--(4,4.25);
            	\draw[dashed,red] (2,4.25)--(4,4.25);
            	\draw[dashed,red] (2,4.25)--(2,5.25);
            	\draw[dashed,red] (0,5.25)--(2,5.25);
            	
            	%% PLOT of W
            	\draw[thick, blue] (0,5.25)--(2,5.25);
            	\draw[thick, blue] plot[smooth, tension=1] coordinates {(2,5.25) (2.5,4.60) (4,4.25)};
            	\draw[thick, blue] plot[smooth, tension=1] coordinates {(4,4.25) (4.5,3.60) (6,3.25)};
            	\draw[thick, blue] plot[smooth, tension=1] coordinates {(6,3.25) (6.5,1.9) (7.5,1.25)};
            	\draw[thick, blue] plot[smooth, tension=1] coordinates {(7.5,1.25) (8,.5) (9,.25)};
            	\draw[thick,blue] (9,.25)--(11,.25);
            	
            	%% MARKERS 
            	\draw[thick,black, |-|] (9,1.25)--(9,3.25);
            	\draw (9,2.25) node[anchor=west] {\Large$ \left(\frac{\eta}{k^2}\right)^{c+\textcolor{blue}{\frac12}}$};
            	
            	\draw[thick,green!70!black, |-|] (9,5.25)--(9,4.25);
            	\draw (9,4.75) node[anchor=west] {\Large$ \left(\frac{\eta}{k^2}\right)^{c}$};
            	
            	\draw (2,2) node {\large $w^{-1}(t,k,\eta)$};
            		
            	\draw[thick,red] (-.1,5.25)--(.1,5.25);
            	\draw (0,5.25) node[anchor=east] { $\mathrm{e}^{\sqrt{\eta}}$};
            	
            	\draw[thick,red] (-.1,.25)--(.1,.25);
            	\draw (0,.25) node[anchor=east] { $1$};
          
                \draw[thick,red] (2,-.1)--(2,.1);
                \draw (2,0) node[anchor=north] { $\sqrt{\eta}$}; 
                \draw[thick,red] (4,-.1)--(4,.1);
                \draw (4,0) node[anchor=north] { $\frac{\eta}{k+1}$};
                \draw[thick,red] (6,-.1)--(6,.1);
                \draw (6,0) node[anchor=north] { $\red{\frac{\eta}{k}}$}; 
                \draw[thick,red] (7.5,-.1)--(7.5,.1);
                \draw (7.5,0) node[anchor=north] { $\frac{\eta}{k-1}$};  
                \draw[thick,red] (9,-.1)--(9,.1);
                \draw (9,0) node[anchor=north] { $2\eta $}; 
                
                \draw (11.5,0) node[anchor=north east] {$t$};
                \fill [red, opacity=.5] (5,-.1) rectangle (7,.1);
                
        \end{tikzpicture}
\end{center}
The actual weight $A(t,k,\eta)$ used is composed of three main pieces,
\begin{equation}
A(t,k,\eta) =\e^{\lambda(t)(|k|+|\eta|)^s}(m^{-1} w^{-1})(t,k,\eta) +\text{technical corrections},
\end{equation}
where $w$ is the weight based on the toy model. The exponent $0<\lambda_{\infty}\leq\lambda(t)\leq \lambda_0$ is decreasing continuosly in time, allowing for a loss of an infinite numbers of derivatives but is always comparable to a Gevrey-$1/s$ space with different radius of regularity $\lambda$. The weight $m$ is instead introduced to handle the linear error terms, because in the nonlinear argument we cannot hope to argue at fixed $k,\eta$ frequencies as in Section \ref{sec:linear}. The definition of $m$ already appeared in other related problems and its main property is that $\de_t m/m\approx (1+(t-\eta/k)^2)^{-1}$. 

With the weight $A$ at hand, we need to define the energy functional. The first guess is to directly adapt the one that already works in the linear case, see \eqref{def:pointwise-functional-Couette}, namely
\begin{equation}
E_L(t)=\frac12\left(\|(AZ)(t)\|^2_{L^2}+\|(AQ)(t)\|^2_{L^2}+\frac{1}{2\beta}\jap{\left(\frac{\de_t p}{|k|^\frac12 p}AZ\right)(t),(AQ)(t)}_{L^2}\right).
\end{equation}
The goal is  to bootstrap the estimate $E_L(t)\lesssim \eps^{2}$ for all $t\lesssim \eps^{-2}$, from which the bounds in Theorem \ref{thm:main} would follow by going back to the original variables (for the lower bound an additional Duhamel argument is needed). Unfortunately, due to errors related to $\theta_0, u_0$, the control of this functional is not sufficient to close the bootstrap. Indeed, the variables $Z,Q$ do not have enough regularity to close the required estimates on the zero modes ($Z$
is essentially half derivative of the velocity field). To overcome this problem, the \textit{natural energy} 
\begin{equation}
E_n(t):=\frac12\left(\|(A\Omega)(t)\|^2_{L^2}+\beta^2\|(A\nabla_L\Theta)(t)\|^2_{L^2}\right)
\end{equation}
is introduced. This would in fact be the first energy functional to try if one aims at proving the bounds in Theorem \ref{thm:main}. Notice that this energy functional is at the highest level of regularity and, due to the growth observed in the linearized problem, we expect that $E_n(t)\lesssim \eps^2 t$. With the last bound at hand, the estimates in Theorem \ref{thm:main} would again follow and therefore the introduction of $Z,Q$ seems useless. However, the time-derivative of $E_n$ introduces a linear error whose control is equivalent at having the bounds on $AQ$, see \cite{BBCZD21}*{Remark 2.6}.

There is also another energy functional related to the change of coordinates which we do not discuss here, see \cite{BBCZD21}*{Remark 2.5}. This energy functional bears similarities with the one introduced in \cite{BM15} for the change of coordinates. However, in \cite{BBCZD21} a better control of the $x$-averages is needed.

\subsection{The bootstrap argument}
When taking the time derivative of one of the energy functionals, terms with a good sign appear from the time-derivative of the weight $A$, as in \eqref{eq:dtwz}, sometimes called \textit{Cauchy-Kovalevskaya} terms. However, several nonlinear error terms are generated. Even in the toy model, there is a significant disparity in the behavior of the velocity field and vorticity/density across different frequency regimes. Hence, one has to split the nonlinear terms to account for the different regimes
(essentially a paraproduct decomposition with more frequency regions than usual). 

First, for some terms one can exploit the transport structure; in these cases, commutators appear. For instance, for any incompressible vector field $\boldsymbol{V}$ and function $G$ we have 
\begin{equation}
\jap{A(\boldsymbol{V} \cdot \nabla G),AG}_{L^2}=\jap{[A,\boldsymbol{V}] \cdot \nabla G,AG}_{L^2},
\end{equation}
where $[A,\boldsymbol{V}]$ is the commutator. Writing the integral explicitly, the integrand arising from the commutator is $$(A(k,\eta)-A(\ell,\xi))\hat{\boldsymbol{V}}(k-\ell,\eta-\xi),$$
in Fourier space. This structure is particularly useful when $\boldsymbol{V}$ is regular enough to allow some loss of derivatives. This will correspond to the low frequency (both in $k,\eta$) piece of the velocity field $\bU_{\neq}$, generating errors that are usually called \textit{transport nonlinearities}, see \cite{BBCZD21}*{Section 6.3}. The key point here is to gain derivatives, time-factors or reconstruct pieces of $\de_tA/A$ from the commutator. To explain this last point heuristically, since $A$ contains weights depending on $t-\eta/k$, doing a commutator in $\eta$ at fixed $k$ looks like doing a time derivative of $A$ thanks to the mean value theorem. For instance, assuming that $A(t,k,\eta)=A(t-\eta/k)$, by the mean value theorem we get
\begin{equation}
|A(t,k,\eta)-A(t,k,\xi)|= |(\de_\eta A)(t,k,\tilde{\xi})||\eta-\tilde{\xi}|=\frac{1}{|k|} |(\de_t A)(t,k,\tilde{\xi})||\eta-\tilde{\xi}|.
\end{equation}
A commutator is giving us back a derivative in the horizontal variable at the price of paying $\eta-\tilde{\xi}$ derivatives (associated to the low-frequency velocity field) plus a control on $\de_tA$. But notice that $\de_tA$ is itself related to the commutator $[\de_t,A]$, which now is helping us in absorbing similar errors. Unfortunately, the weights do not behave this nicely when exchanging $k$'s (which are  discrete frequencies as well, so the mean value theorem cannot directly be applied). This is the main reason why a careful understanding of the ``price" to exchange frequencies in the weights is needed, see \cite{BBCZD21}*{Section 4}.

When instead the velocity field $\bU_{\neq}$ is at high frequencies, commutators does not help because the gradient in the transport operator hits a low-frequency term. Here is where the design of the weight $w$ plays a crucial role. Indeed, the toy model in Section \ref{sec:toy} is constructed exactly with a term arising in this regime. These nonlinear errors are usually called \textit{reaction nonlinearities}, whose detailed treatment can be found in \cite{BBCZD21}*{Section 6.4}.

In the end, the detailed computations of the bounds for a specific component in the nonlinear errors are not particularly challenging from a mathematical perspective. These bounds rely on non-sophisticated algebraic and Cauchy-Schwartz inequalities. The primary difficulties lie in correctly deducing the appropriate toy model and in devising the weights to maintain favorable frequency exchange properties. Then, there is a substantial number of error terms to manage, but with the appropriate weights in place, the only other essential element is to possess the patience to rigorously assess each component.
 
\label{sec:energies}

\section{Viscous asymptotic stability}\label{sec:viscous}

Understanding the conditions and mechanisms that lead to the transition from laminar to turbulent flow represents a fundamental challenge within the realm of fluid mechanics. This quest dates back to the groundbreaking experiments conducted by Reynolds in 1883 and has since fueled extensive research efforts in the field of hydrodynamic stability. A key focus of this research has been the exploration of critical flow parameters, notably the Reynolds number, which denote the onset of turbulent behavior.

\subsection{Nonlinear stability thresholds}
In the context of equations \eqref{eq:omNLintro}-\eqref{eq:thNLintro}, where we consider the simplified case of $\nu=\kappa>0$ for clarity, the nonlinear transition threshold problem can be precisely formulated as follows: Given two function spaces denoted as $X$ and $Y$, the objective is to determine the minimal value of $\alpha\geq0$ such that the following condition holds:
\begin{align}\label{eq:stab}
\| \omega^{in},\theta^{in}\|_X\ll \nu^\alpha \qquad \Rightarrow \qquad
\begin{cases}
\| \omega(t),\theta(t)\|_Y\ll 1, \quad \forall t>0,\\
\| \omega(t),\theta(t)\|_Y \to 0,\quad \text{as } t\to\infty.
\end{cases}
\end{align}
Here, the exponent $\alpha$ serves as a crucial indicator, delineating the size of the basin of attraction for the Couette flow. A first estimate can be traced back to \cite{Zillinger21}, with subsequent refinements in \cite{ZZ23}, as stated below.

\begin{theorem}[Nonlinear stability threshold, \cite{ZZ23}]
\label{th:zhao}
Let $s\geq 6$ and $\beta^2>1/4$. There exist $\eps_0,\nu_0\in(0,1)$ such that for all $\nu\in(0,\nu_0)$, $\eps\in(0,\eps_0)$ and 
all initial data complying with
$$
\|\bu^{in}\|_{H^{s+1}}+\|\theta^{in}\|_{H^s+2}\leq \eps \nu^{\frac12},
$$
then the solution to \eqref{eq:omNLintro}-\eqref{eq:thNLintro} satisfies for all $t\geq 0$ the enhanced dissipation estimates
\begin{align}
\|u^x_{\neq}(t)\|_{L^2}+\l t\r\|u^y(t)\|_{L^2}+\l t\r^{-1}\|\omega_{\neq}(t)\|_{L^2}+\|\theta_{\neq}(t)\|_{L^2}&\lesssim \eps\nu^{\frac12}\l t\r^{-\frac12}\e^{-c_0\nu^{\frac13}t}\\
\nu^{-\frac14}\|u^x_0(t)\|_{L^2}+\|\omega_0(t)\|_{L^2}+\nu^{-\frac14}\|\theta_0(t)\|_{L^2}&\lesssim \eps\nu^{\frac14},
\end{align}
where $c_0>0$ is a constant independent of $t$, $\eps$ and $\nu$.
\end{theorem}
This result effectively combines the inviscid and viscous dynamics found in the linearized problem, see Theorems \ref{thm:damplin} and \ref{thm:enhancelin}. The proof of this theorem requires weighted energy estimates and it crucially relies on the use of symmetric variables (slightly modified for the density), highlighting again the flexibility of this scheme. 

The use of Sobolev regularity rather than Gevrey is expected from the regularizing properties of the dissipation, which allows to trade regularity for time-decay for instance. 
The exponent $\alpha=1/2$ for the threshold in  Theorem \ref{th:zhao} is likely the minimal value one can hope for. Indeed, the best available result for the 2D homogeneous case is $\alpha=1/3$ (likely optimal) in \cite{MZ22stab}. The $\sqrt{t}$-instability of the linearized problem results in a transient growth of order $\nu^{-\frac16}$. To compensate for this growth, it seems natural to impose an extra $\nu^{\frac16}$ smallness of the initial data, giving us the $1/3+1/6=1/2$ threshold in Theorem \ref{th:zhao}.  Finally, we believe that since the symmetrization scheme captures the linearized dynamics in a quite optimal way, it is useful in obtaining  ``optimal" (in the heuristic sense explained above) estimates for transition threshold, see the 2D MHD case as well \cite{Dolce23}.

\subsection{The non-diffusive case}

As one can infer from the linear theory (\emph{c.f.} Theorem \ref{thm:nondiffmasmoudi}), an analogous result in the viscous, non-diffusive case $\nu>0$, $\kappa=0$, is not really expected. Nonetheless, a form of asymptotic stability still holds when the perturbation is assumed to be in Gevrey-class, as the main result of \cite{masmoudi20bouss} states.

\begin{theorem}[Viscous asymptotic stability, \cite{masmoudi20bouss}]
Let $\nu=1$, $\kappa=0$, $\beta>0$, $s\in(1/3,1]$ and $\lambda_0>0$. There exists $\eps_0\in(0,1)$ such that for all $\eps\in(0,\eps_0)$ and 
all initial data with zero $x$-average complying with
$$
\|\omega^{in}\|_{\mathcal{G}^{\lambda_0,s}}+\|\theta^{in}\|_{\mathcal{G}^{\lambda_0,s}}\leq \eps,
$$
then the solution to \eqref{eq:omNLintro}-\eqref{eq:thNLintro} satisfies asymptotic stability estimate
\begin{align}
\label{eq:uom}
\l t\r\|u^x_{\neq}(t)\|_{L^2}+\l t\r^2\|u^y(t)\|_{L^2}+\|\omega_{\neq}(t)\|_{L^2}&\lesssim \eps \l t\r^{-2},
\end{align}
for all $t\geq 0$.
\end{theorem}
The bounds in \eqref{eq:uom} are in agreement with the linearized case treated in Theorem \ref{thm:nondiffmasmoudi}. We are not including a full statement only to highlight the main differences with Theorems \ref{thm:main} and \ref{th:zhao}. However, a crucial fact here is that the density does not decay but it still is asymptotically stable (as the vorticity in 2D homogeneous Euler). To give a heuristic explanation of why this happens, consider the good unknown $\Sigma$ introduced in Section \ref{sec:zerodiff}. The fact that we can bound it uniformly, suggests that we can effectively replace the vorticity with a quantity scaling like $\de_x\Delta^{-1}\theta$. Hence, the density behaves as if it were satisfying the 2D Euler equation with a modified Biot-Savart law, where the velocity field $\bu=\nabla^{\perp}\Delta^{-1}\omega$ is replaced by $\nabla^\perp \de_x\Delta^{-2} \theta$. We do not delve more deeply into a more detailed explanation of these mechanisms, and we refer to \cite{masmoudi20bouss} for more details.

\section{Conclusions and open problems}

The natural progression from the work presented in \cite{BBCZD21} involves its expansion to encompass the full system of stratified fluids, eliminating the Boussinesq assumption which presumes constant density unless directly contributing to buoyant forces. This extension is of paramount physical relevance, particularly when considering exponential stratification scenarios. At the linearized level, the corresponding analysis was conducted in \cite{BCZD20}, yielding analogous outcomes concerning linear inviscid damping and buoyancy instability. Nonetheless, the nonlinear case proves to be a challenging extension, as density now assumes a pivotal role in its interaction with the momentum equation. These findings can also be situated within the context of the global existence for non-homogeneous Euler equations, where inviscid damping has led to the emergence of examples of global solutions \cites{CWZZ23,Zhao23}.

Theorem \ref{thm:main} describes the dynamics up to the nontrivial time-scale $O(\eps^{-2})$, when the system possibly exits the linearized regime. In particular, thanks to the Gevrey regularity assumptions on the initial data, echoes are under control. It would be interesting to characterize the dynamics after this time and understand whether an asymptotic state near Couette is reached or (secondary) instabilities kick in and drive the system really away from equilibrium. In the latter case, the long-time dynamics may just be too complicated to understand with current methods, but at the moment it is even unclear if the system really leaves the perturbative regime.

\bibliographystyle{siam}
\bibliography{schwartzBiblio}

\begin{bibdiv}
\begin{biblist}

\bib{antonelli2021linear}{article}{
      author={Antonelli, Paolo},
      author={Dolce, Michele},
      author={Marcati, Pierangelo},
       title={Linear stability analysis of the homogeneous {C}ouette flow in a
  2{D} isentropic compressible fluid},
        date={2021},
        ISSN={2524-5317,2199-2576},
     journal={Ann. PDE},
      volume={7},
      number={2},
       pages={Paper No. 24, 53},
         url={https://doi.org/10.1007/s40818-021-00112-3},
      review={\MR{4329964}},
}

\bib{BBCZD21}{article}{
      author={{Bedrossian}, Jacob},
      author={{Bianchini}, Roberta},
      author={{Coti Zelati}, Michele},
      author={{Dolce}, Michele},
       title={{Nonlinear inviscid damping and shear-buoyancy instability in the
  two-dimensional Boussinesq equations}},
        date={2021-03},
     journal={Comm. Pure Appl. Math., to appear},
      eprint={2103.13713},
}

\bib{BM15}{article}{
      author={Bedrossian, Jacob},
      author={Masmoudi, Nader},
       title={Inviscid damping and the asymptotic stability of planar shear
  flows in the 2{D} {E}uler equations},
        date={2015},
     journal={Publ. Math. Inst. Hautes \'{E}tudes Sci.},
      volume={122},
       pages={195\ndash 300},
         url={https://doi.org/10.1007/s10240-015-0070-4},
}

\bib{BCZD20}{article}{
      author={Bianchini, Roberta},
      author={Coti~Zelati, Michele},
      author={Dolce, Michele},
       title={Linear inviscid damping for shear flows near {C}ouette in the
  2{D} stably stratified regime},
        date={2022},
        ISSN={0022-2518},
     journal={Indiana Univ. Math. J.},
      volume={71},
      number={4},
       pages={1467\ndash 1504},
         url={https://doi.org/10.1512/iumj.2022.71.9040},
      review={\MR{4481091}},
}

\bib{Buhler2009}{book}{
      author={B\"uhler, O.},
       title={Waves and mean flows},
   publisher={Cambridge University Press},
        date={2009},
}

\bib{CWZZ23}{article}{
      author={{Chen}, Qi},
      author={{Wei}, Dongyi},
      author={{Zhang}, Ping},
      author={{Zhang}, Zhifei},
       title={{Nonlinear inviscid damping for 2-D inhomogeneous incompressible
  Euler equations}},
        date={2023-03},
     journal={arXiv e-prints},
      eprint={2303.14858},
}

\bib{CKRZ08}{article}{
      author={Constantin, P.},
      author={Kiselev, A.},
      author={Ryzhik, L.},
      author={Zlato\v{s}, A.},
       title={Diffusion and mixing in fluid flow},
        date={2008},
        ISSN={0003-486X,1939-8980},
     journal={Ann. of Math. (2)},
      volume={168},
      number={2},
       pages={643\ndash 674},
         url={https://doi.org/10.4007/annals.2008.168.643},
      review={\MR{2434887}},
}

\bib{CZDZ23}{article}{
      author={{Coti Zelati}, Michele},
      author={{Del Zotto}, Augusto},
       title={{Suppression of lift-up effect in the 3D Boussinesq equations
  around a stably stratified Couette flow}},
        date={2023-09},
     journal={arXiv e-prints},
      eprint={2309.06426},
}

\bib{CZN23strip}{article}{
      author={{Coti Zelati}, Michele},
      author={{Nualart}, Marc},
       title={{Explicit solutions and linear inviscid damping in the
  Euler-Boussinesq equation near a stratified Couette flow in the periodic
  strip}},
        date={2023-09},
     journal={arXiv e-prints},
       pages={arXiv:2309.08419},
      eprint={2309.08419},
}

\bib{CZN23chan}{article}{
      author={{Coti Zelati}, Michele},
      author={{Nualart}, Marc},
       title={{Limiting absorption principles and linear inviscid damping in
  the Euler-Boussinesq system in the periodic channel}},
        date={2023-09},
     journal={arXiv e-prints},
       pages={arXiv:2309.08445},
      eprint={2309.08445},
}

\bib{dauxois18}{incollection}{
      author={Dauxois, T.},
      author={Joubaud, S.},
      author={Odier, P.},
      author={Venaille, A.},
       title={Instabilities of internal gravity wave beams},
        date={2018},
   booktitle={Annual review of fluid mechanics. {V}ol. 50},
      series={Annu. Rev. Fluid Mech.},
      volume={50},
   publisher={Annual Reviews, Palo Alto, CA},
       pages={131\ndash 156},
      review={\MR{3753206}},
}

\bib{deng2018long}{article}{
      author={Deng, Yu},
      author={Masmoudi, Nader},
       title={Long-time instability of the {Couette} flow in low {Gevrey}
  spaces},
        date={2023},
     journal={Communications on Pure and Applied Mathematics},
      volume={76},
       pages={2804\ndash 2887},
}

\bib{Dolce23}{article}{
      author={{Dolce}, Michele},
       title={{Stability threshold of the 2D Couette flow in a homogeneous
  magnetic field using symmetric variables}},
        date={2023-08},
     journal={arXiv e-prints},
      eprint={2308.12589},
}

\bib{drazin1981}{book}{
      author={Drazin, P.~G.},
      author={Reid, W.~H.},
       title={Hydrodynamic stability},
     edition={Second},
      series={Cambridge Mathematical Library},
   publisher={Cambridge University Press, Cambridge},
        date={2004},
        ISBN={0-521-52541-1},
         url={https://doi.org/10.1017/CBO9780511616938},
        note={With a foreword by John Miles},
      review={\MR{2098531}},
}

\bib{DRAZIN19661}{incollection}{
      author={Drazin, P.G.},
      author={Howard, L.N.},
       title={Hydrodynamic stability of parallel flow of inviscid fluid},
        date={1966},
      editor={Chernyi, G.G.},
      editor={Dryden, H.L.},
      editor={Germain, P.},
      editor={Howarth, L.},
      editor={Olszak, W.},
      editor={Prager, W.},
      editor={Probstein, R.F.},
      editor={Ziegler, H.},
      series={Advances in Applied Mechanics},
      volume={9},
   publisher={Elsevier},
       pages={1\ndash 89},
  url={https://www.sciencedirect.com/science/article/pii/S0065215608700061},
}

\bib{gallay2019stability}{article}{
      author={Gallay, Thierry},
       title={Stability of vortices in ideal fluids: the legacy of kelvin and
  rayleigh},
        date={2020},
     journal={Hyperbolic Problems: Theory, Numerics, Applications. Proceedings
  of the XVII International Conference on Hyperbolic Problems (HYP2018)},
       pages={690},
}

\bib{Hartman75}{article}{
      author={Hartman, R.J.},
       title={Wave propagation in a stratified shear flow},
        date={1975},
     journal={J. Fluid Mech.},
      volume={71},
       pages={89\ndash 104},
}

\bib{howard1961note}{article}{
      author={Howard, Louis~N},
       title={Note on a paper of {John W. Miles}},
        date={1961},
     journal={Journal of Fluid Mechanics},
      volume={10},
      number={4},
       pages={509\ndash 512},
}

\bib{masmoudi20bouss}{article}{
      author={Masmoudi, Nader},
      author={Said-Houari, Belkacem},
      author={Zhao, Weiren},
       title={Stability of the {C}ouette flow for a 2{D} {B}oussinesq system
  without thermal diffusivity},
        date={2022},
        ISSN={0003-9527},
     journal={Arch. Ration. Mech. Anal.},
      volume={245},
      number={2},
       pages={645\ndash 752},
         url={https://doi.org/10.1007/s00205-022-01789-x},
      review={\MR{4451473}},
}

\bib{MZ22stab}{article}{
      author={Masmoudi, Nader},
      author={Zhao, Weiren},
       title={Stability threshold of two-dimensional {C}ouette flow in
  {S}obolev spaces},
        date={2022},
        ISSN={0294-1449,1873-1430},
     journal={Ann. Inst. H. Poincar\'{e} C Anal. Non Lin\'{e}aire},
      volume={39},
      number={2},
       pages={245\ndash 325},
         url={https://doi.org/10.4171/aihpc/8},
      review={\MR{4412070}},
}

\bib{miles1961stability}{article}{
      author={Miles, John~W},
       title={On the stability of heterogeneous shear flows},
        date={1961},
     journal={Journal of Fluid Mechanics},
      volume={10},
      number={4},
       pages={496\ndash 508},
}

\bib{rieutord2014fluid}{book}{
      author={Rieutord, Michel},
       title={Fluid dynamics: an introduction},
   publisher={Springer},
        date={2014},
}

\bib{YL18}{article}{
      author={Yang, Jincheng},
      author={Lin, Zhiwu},
       title={Linear inviscid damping for {C}ouette flow in stratified fluid},
        date={2018},
     journal={J. Math. Fluid Mech.},
      volume={20},
      number={2},
       pages={445\ndash 472},
         url={https://doi.org/10.1007/s00021-017-0328-3},
}

\bib{ZZ23}{article}{
      author={Zhai, Cuili},
      author={Zhao, Weiren},
       title={Stability threshold of the {C}ouette flow for {N}avier-{S}tokes
  {B}oussinesq system with large {R}ichardson {N}umber {$\gamma^2 > \frac
  14$}},
        date={2023},
        ISSN={0036-1410,1095-7154},
     journal={SIAM J. Math. Anal.},
      volume={55},
      number={2},
       pages={1284\ndash 1318},
         url={https://doi.org/10.1137/22M1495160},
      review={\MR{4580336}},
}

\bib{Zhao23}{article}{
      author={{Zhao}, Weiren},
       title={{Inviscid damping of monotone shear flows for 2D inhomogeneous
  Euler equation with non-constant density in a finite channel}},
        date={2023-04},
     journal={arXiv e-prints},
      eprint={2304.09841},
}

\bib{ZillingerBouss}{article}{
      author={Zillinger, Christian},
       title={On enhanced dissipation for the {B}oussinesq equations},
        date={2021},
        ISSN={0022-0396},
     journal={J. Differential Equations},
      volume={282},
       pages={407\ndash 445},
         url={https://doi.org/10.1016/j.jde.2021.02.029},
}

\bib{Zillinger21}{article}{
      author={Zillinger, Christian},
       title={On enhanced dissipation for the {B}oussinesq equations},
        date={2021},
        ISSN={0022-0396,1090-2732},
     journal={J. Differential Equations},
      volume={282},
       pages={407\ndash 445},
         url={https://doi.org/10.1016/j.jde.2021.02.029},
      review={\MR{4219325}},
}

\end{biblist}
\end{bibdiv}

\end{document}